\documentclass[11pt]{article}

\usepackage[top=1.3in,bottom=1.3in,left=1.31in,right=1.31in]{geometry}
\usepackage{amsmath,amssymb,amsthm,amsfonts,verbatim,hyperref}
\usepackage{graphicx,caption,bbm}
\usepackage{enumerate}
\usepackage[all,2cell]{xy}

\theoremstyle{plain}
\newtheorem{thm}{Theorem}[section]
\newtheorem{cor}[thm]{Corollary}
\newtheorem{prop}[thm]{Proposition}

\newtheorem{lemma}[thm]{Lemma}

\theoremstyle{definition}
\newtheorem{rem}[thm]{Remark}

 \DeclareMathOperator{\SO}{SO}

\DeclareMathOperator{\SU}{SU} 
\DeclareMathOperator{\Sp}{Sp}

 \DeclareMathOperator{\Mat}{M}

\DeclareMathOperator{\Isom}{Isom}

\DeclareMathOperator{\Jac}{Jac}
\DeclareMathOperator{\Vol}{Vol}
\DeclareMathOperator{\Id}{Id}
\DeclareMathOperator{\hd}{hd}
\DeclareMathOperator{\bary}{bar}\DeclareMathOperator{\rank}{rank}
\DeclareMathOperator{\Fe}{F}\DeclareMathOperator{\meas}{{\cal M}}

\newcommand{\eps}{\varepsilon}

\newcommand{\ga}{\gamma}
\newcommand{\Ga}{\Gamma}

\newcommand{\la}{\lambda}

\newcommand{\grad}{\nabla}

\newcommand{\iny}{\infty}

\newcommand{\iso}{\cong}

\newcommand{\innp}[1]{\left< #1 \right>}
\newcommand{\abs}[1]{\left\vert#1\right\vert}
\newcommand{\norm}[1]{\left\vert\left\vert #1\right\vert\right\vert}
\newcommand{\set}[1]{\left\{#1\right\}}

\newcommand{\pr}[1]{\left( #1 \right) }
\newcommand{\su}{\subset}

\newcommand{\lra}{\longrightarrow}

\newcommand{\B}[1]{\ensuremath{\mathbf{#1}}}
\newcommand{\BB}[1]{\ensuremath{\mathbb{#1}}}
\newcommand{\Cal}[1]{\ensuremath{\mathcal{#1}}}

\newcommand{\Hy}{\ensuremath{\B{H}}}
\newcommand{\N}{\ensuremath{\B{N}}}
\newcommand{\Q}{\mathbb{Q}}
\newcommand{\Z}{\mathbb{Z}}
\newcommand{\R}{\mathbb{R}}
\newcommand{\C}{\mathbb{C}}
\newcommand{\F}{\mathbb{K}}

\let\oldmarginpar\marginpar
\renewcommand\marginpar[1]{\-\oldmarginpar[\raggedleft\footnotesize #1]%
{\raggedright\footnotesize #1}}

\begin{document}

%---------------------------------------------------------------------------
%---------------------------------------------------------------------------
%%%%    TITLE                    %%%%%%%%%%%%%%%%%
%---------------------------------------------------------------------------
%---------------------------------------------------------------------------

\title{\textbf{A vanishing theorem for the homology of \\ discrete subgroups of $\mathrm{Sp}(n,1)$ and $\mathrm{F}_4^{-20}$}}
\author{Chris Connell\thanks{Indiana University, Bloomington, IN. E-mail \tt{connell@indiana.edu}},~~Benson Farb\thanks{University of Chicago, Chicago, IL. E-mail: \tt{farb@math.uchicago.edu}}, and D. B. McReynolds\thanks{Purdue University, West Lafayette, IN. E-mail: \tt{dmcreyno@purdue.edu}}}
\maketitle

%---------------------------------------------------------------------------
%---------------------------------------------------------------------------
%%%%     ABSTRACT             %%%%%%%%%%%%%%%%
%---------------------------------------------------------------------------
%---------------------------------------------------------------------------

\begin{abstract}
For any discrete, torsion-free subgroup $\Gamma$ of $\mathrm{Sp}(n,1)$ (resp.\ $\mathrm{F}_4^{-20}$) with no parabolic elements, we prove that $H_{4n-1}(\Gamma;V)=0$ (resp.\ $H_i(\Gamma;V)=0$ for $i=13,14,15$) for any $\Gamma$--module $V$. The main technical advance is a new bound on the $p$--Jacobian of the barycenter map of Besson--Courtois--Gallot.  We also apply this estimate to obtain an inequality between the critical exponent and homological dimension of $\Gamma$, improving on work of M.~Kapovich.
\end{abstract}

%---------------------------------------------------------------------------
%---------------------------------------------------------------------------
%%%%         INTRODUCTION          %%%%%%%%%%%%%
%---------------------------------------------------------------------------
%---------------------------------------------------------------------------
\section{Introduction}

Homological vanishing results for lattices $\Gamma$ in semisimple Lie groups $G$ go back to Selberg, Weil, Matsushima, S.~P.~Wang and others; see \S 7.68 of \cite{Raghunathan} for a brief history up to 1972.  One culmination of that enterprise is Margulis's theorem  (see \cite[Cor.~IX.5.8]{Margulis}) that $H^1(\Gamma;V)=0$ for any irreducible lattices $\Gamma$ of $G$ and any finite-dimensional $\Gamma$--module $V$, provided $\rank_\R(G)>1$.  Starkov \cite[Thm 2]{Starkov}, building on work of many people, extended this vanishing to lattices $\Gamma$ in the rank $1$ groups $\Sp(n,1),n\geq 2$ and $\Fe_4^{-20}$.  

The above results are proved using either Hodge theory (or more generally harmonic maps) or ergodic theory.  Since each of these tools requires in a crucial way that the measure of $G/\Gamma$ be finite, a new idea is needed if one wishes to extend homology vanishing to discrete subgroups $\Gamma<G$ that have infinite covolume.  

In this paper we find a different mechanism for homology vanishing and we apply it to infinite covolume discrete subgroups $\Gamma$ in 
$\Sp(n,1),n\geq 2$ and $\Fe_4^{-20}$. If $M$ is the locally symmetric manifold associated to $\Gamma$, we construct a smooth map that is homotopic to the identity on $M$ and contracts $p$--dimensional volume at every point $x\in M$. The existence of such a map provides a powerful tool for establishing the vanishing of $H_p(\Gamma;V)$ for arbitrary coefficients $V$; see Corollary \ref{SmallJacCorollary} below. 

\subsection{Results}
To state our results, let $G$ be a simple Lie group with $\rank_\R(G)=1$; namely, $\SO(n,1)$, $\SU(n,1)$, $\Sp(n,1)$, or $\Fe_4^{-20}$. These groups are, up to isogeny, the groups of orientation-preserving isometries of the real, complex, and quaternionic-hyperbolic spaces $\Hy_\R^n,\Hy_\C^n,\Hy_\BB{H}^n$, and the Cayley-hyperbolic plane $\Hy_\BB{O}^2$, respectively.  Throughout, we let $\F=\R, \C, \BB{H}$ or $\BB{O}$ and set $d:=\dim_\R\F$, so that $d=1,2,4,$ or $8$. Each $\F$--hyperbolic space $\Hy_\F^n$ is a connected, contractible, negatively curved, symmetric Riemannian manifold of dimension $dn$ with normalized sectional curvatures between $-4$ and $-1$.  To any discrete, torsion-free subgroup $\Gamma<G$, we have the corresponding locally symmetric manifold $M_\Gamma:=\Hy_\F^n/\Gamma$.  Note that since $\Hy_\F^n$ is contractible, there is an isomorphism $H_i(M_\Gamma;V)\iso H_i(\Gamma;V)$ for any coefficient system $V$. Our main result is the following.

\begin{thm}[{\bf Homology vanishing theorem}]
\label{GenCDGapnew}
Let $\Gamma$ be any discrete, finitely-generated, torsion-free subgroup of $\Sp(n,1)$ or $\Fe_4^{-20}$. Assume that $\Vol(\Hy^n_\F/\Gamma)=\infty$ and that $\Gamma$ has no parabolic elements.  
\begin{enumerate}
\item[(a)] If $\Gamma<\Sp(n,1)$ and $n\geq 2$ then for all $\Gamma$--modules $V$: 
\begin{equation}\label{QuatHomVanEq}
H_{4n-1}(\Gamma;V)=0. 
\end{equation}

\item[(b)] If $\Gamma<\Fe_4^{-20}$ then for all $\Gamma$--modules $V$:
\begin{equation}\label{CayHomVanEq}
H_{13}(\Gamma;V)=H_{14}(\Gamma;V)=H_{15}(\Gamma;V)=0.
\end{equation}
\end{enumerate}
\end{thm}

The statement of Theorem~\ref{GenCDGapnew} also holds when $\Gamma$ is cocompact and $V$ is finite-dimensional; apply Poincar\'e Duality to the above results from the first paragraph. 

Theorem~\ref{GenCDGapnew} has the following immediate consequence for the homological dimension $\hd(\Gamma)$ of the groups $\Gamma$.

\begin{cor}[{\bf Homological dimension gap}]
\label{corollary:gap}
If $\Gamma<\Sp(n,1)$ is any discrete, finitely-generated, torsion-free subgroup with no parabolic elements, then either 
$\hd(\Gamma)=4n$ or $\hd(\Gamma)\leq 4n-2$. Moreover, the value $4n$ is obtained precisely when $\Gamma$ is a lattice. Similarly, if $\Gamma < \Fe_4^{-20}$ satisfies the same conditions, then either $\hd(\Gamma) = 16$ or $\hd(\Gamma) \leq 12$. Moreover, the value $16$ is obtained precisely when $\Gamma$ is a lattice.
\end{cor}

We discuss Corollary \ref{corollary:gap} in \cite{CFM} in the broader context of ``homological dimension spectrum'' of a finitely-generated group.  

\begin{rem}
\begin{enumerate}
\quad

\item Carron--Pedon \cite[Cor.~5.6]{CP} proved $H_{4n-1}(M;\Z)=0$ (reps., $H_{15}(M;\Z)=0$) under the same assumptions as Theorem~\ref{GenCDGapnew}.

\item The restrictions on parabolic elements in Theorem~\ref{GenCDGapnew} and Corollary \ref{corollary:gap} are necessary. For example, in $\Sp(n,1)$ (respectively, $\Fe_4^{-20}$) there is a discrete, torsion-free nilpotent group $\Gamma$ of parabolic isometries with $H_i(\Gamma;\Q)\neq 0$ for each $0\leq i\leq 4n-1$ (respectively, $0\leq i\leq 15$). For instance, if $G = KAN$ is the Iwasawa decomposition of $G$, then any torsion-free lattice in the simply connected, connected nilpotent Lie group $N$ provides such an example. 

\item The statements of Theorem~\ref{GenCDGapnew} and Corollary \ref{corollary:gap}
 are false for $G=\SO(n,1)$ and $G=\SU(n,1)$.  For example, for each $n \geq 2$ there are closed, arithmetic, real-hyperbolic manifolds with closed, arithmetic, totally geodesic submanifolds of every codimension. These examples and similar examples in $\SU(n,1)$ are well-known and will be described in \cite{CFM}.
\item Li \cite[Cor.~6.5]{Li} proved that any cocompact $\Gamma<\Sp(n,1), n\geq 2$ has a finite index subgroup $\Gamma'$ with $H_{4n-2}(\Gamma';\C)\neq 0$.

\end{enumerate}
\end{rem}

The proof of Theorem~\ref{GenCDGapnew} uses $\F$--hyperbolic geometry.   It builds in a crucial way on earlier work of Besson--Courtois--Gallot \cite{BCGSchwartz}, Corlette \cite{Corlette}, Gromov \cite{GBC} and M. Kapovich \cite{Kapovich}.  Along the way we obtain new estimates on the barycenter map, as well as a new inequality between the critical exponent and homological dimension of $\Gamma$.  

%---------------------------------------------------------------------------
\subsection{Vanishing cycles and $p$--Jacobians}

For simplicity, we restrict discussion here to the case $\Gamma<\Sp(n,1)$ with $n\geq 2$. For any $p\geq 1$, the $p$--dimensional volume distortion of a map $F\colon M\to M$ is measured pointwise by the \textbf{\boldmath$p$--Jacobian}  of $F$ at $x\in M$ and is given by 
\begin{equation}\label{eq:JacDef}
\Jac_{p}(F,x) :=\sup \Vert dF_x(u_1)\wedge \cdots \wedge dF_x(u_p)\Vert,
\end{equation}
where the supremum is taken over all orthonormal $p$--frames $\{u_1,\ldots ,u_p\}$ in $T_x M$ and the norm on $\Lambda^p(T_{F(x)}M)$ is the standard norm induced by the Riemannian inner product on $T_{F(x)}M$.  

The main steps in proving Theorem~\ref{GenCDGapnew} follow an idea of M. Kapovich (see \cite{Kapovich}):

\medskip
\noindent
{\bf Step 1 (Contracting self-map): } Construct a self-map $F\colon M\to M$, homotopic to the identity, with the property that 
\begin{equation}
\label{equation:shrink}
\abs{\Jac_{p}(F,x)} < 1\ \ \text{for all $x\in M$.}
\end{equation}
The construction of the map $F$ is by far the most subtle and difficult step in the proof.  As we explain in detail below, previously known bounds on $\abs{\Jac_{p}(F,x)}$ are insufficient to establishing \eqref{equation:shrink}.  

\smallskip
\noindent
{\bf Step 2 (Arbitrarily small $(4n-1)$--mass): } As described precisely in \S\ref{S:Norms} below, in any Riemannian manifold $M$ there is a notion of \textbf{mass} $\norm{c}_{mass}$ of a cycle $c$ representing any $\xi\in H_{4n-1}(M;V)$.  One can think of $\norm{c}_{mass}$ as measuring the $(4n-1)$--dimensional volume of $c$.   Applying the self-map $F$ repeatedly 
 gives, by the crucial inequality \eqref{equation:shrink}: 
 \[\lim_{n\to\infty}\norm{F^n(c)}_{mass}=0\]
On the other hand, $\xi=[F^n(c)]$ since $F$ is homotopic to the identity.  Hence, $\xi$ can be represented by cycles of arbitrarily small norm.  

\smallskip
\noindent
{\bf Step 3 (Gromov's Principle): }Gromov's principle (Theorem~\ref{T71} and Corollary~\ref{SmallJacCorollary} below) states that for $\Gamma$ a discrete, finitely-generated, torsion-free subgroup of $\Isom^+(\Hy_\F^n)$ with no parabolic elements, there exists a constant $\theta_{\F,n}$ such that if $c$ is any $p$--cycle with $\norm{c}_{mass}<\theta_{\F,n}$ then $[c]=0\in H_p(\Gamma;V)$.  Thus $\xi=0$, proving Theorem~\ref{GenCDGapnew}. 

\medskip
The strategy for Cayley-hyperbolic $M$ is similar.  The stronger vanishing comes from a greater abundance of directions of negative curvature, allowing us to find self-maps $F$ with $\abs{\Jac_{p}(F,x)}<1$ for all $x\in M$ and $p=13,14,15$.  

%---------------------------------------------------------------------------
\subsection{The barycenter map}

To find the self-map $F$ used in the proof of Theorem~\ref{GenCDGapnew}, we apply the barycenter method of Besson--Courtois--Gallot (see, e.g.\ \cite{BCGRigid, BCGSchwartz}) to the identity map on $M$. We briefly review this  construction.  Let $\meas_1(\partial \Hy_\F^n)$ denote the spaces of (atomless) probability measures on the visual boundary $\partial \Hy_\F^n$ of $\Hy_\F^n$. The discrete subgroup $\Gamma\subset \Isom^+(\Hy_\F^n)$ determines a family of Patterson--Sullivan measures $\{\mu_x\}_{x\in \Hy_\F^n}$ in $\meas_1(\partial \Hy_\F^n)$. The measure $\mu_x$ encodes the density of points in a $\Gamma$--orbit viewed as living in the visual sphere based at $x$ (see for example Bishop--Jones \cite{BJ} or Sullivan \cite{Sullivan}). 

Attached to any measure $\nu\in\meas_1(\partial \Hy_\F^n)$ is its \textbf{barycenter} $\bary(\nu)\in\Hy_\F^n$, which is the center of mass for the measure measured relative to $\nu$. The point $\bary(\nu)$ is defined as the unique point in $\Hy_\F^n$ with smallest $\nu$--average distance to $\partial\Hy_\F^n$ where distance is measure via horospherical level.  

For any continuous map $f\colon M\to N$ with non-zero degree between $\F$--hyperbolic $n$--manifolds, we can lift $f$ to $\widetilde{f}\colon \Hy_\F^n \to \Hy_\F^n$ and then extend $\widetilde{f}$ to a continuous map $\partial \widetilde{f}\colon \partial\Hy_\F^n \lra \partial\Hy_\F^n$. The map $\partial \widetilde{f}$ induces a pushforward map 
\[ \pr{\partial \widetilde{f}}_\ast\colon \meas_1(\partial \Hy_\F^n) \lra \meas_1(\partial \Hy_\F^n). \] 
Composing the above three maps 
\[ \xymatrix{\Hy_\F^n \ar[rr]^\mu & & \meas_1(\partial \Hy_\F^n) \ar[rr]^{\pr{\partial \widetilde{f}}_\ast} & & \meas_1(\partial \Hy_\F^n) \ar[rr]^{\bary} & & \Hy_\F^n}, \]
we obtain the map $\widetilde{F}\colon \Hy_\F^n \to \Hy_\F^n$ defined by $\widetilde{F}(x):=\bary\pr{\pr{\partial \widetilde{f}}_\ast(\mu_x)}$. An essential feature of both the Patterson--Sullivan and barycenter constructions is that they are canonical. In particular, $\widetilde{F}$ is equivariant with respect to the actions of $\pi_1(M)$ and $\pi_1(N)$, and thus descends to the barycenter map $F\colon M \to N$ homotopic to $f$.

%---------------------------------------------------------------------------
\paragraph{The critical exponent.}

An important feature of $F$ is that $\Jac_{p}(F,x)$ is amenable to explicit computation. $\Jac_{p}(F,x)$ relates to the dynamics of $\Gamma$ through the critical exponent $\delta(\Gamma)$. One of several equivalent definitions of the critical exponent $\delta(\Gamma)$ of $\Gamma$ is given via the Poincar\'{e} series $\sum_{\gamma \in \Gamma} e^{-sd(\gamma p,p)}$ associated to $\Gamma$ and a fixed basepoint $p$ in $\Hy_\F^n$. The \textbf{critical exponent} is defined to be 
\[ \delta(\Gamma):= \inf \set{s~:~\sum_{\gamma \in \Gamma} e^{-sd(\gamma p,p)}<\iny}. \]
The critical exponent $\delta(\Gamma)$ relates to basic invariants of $\Gamma$ and $M$. For example it is equal to the Hausdorff dimension of the conical limit set of $\Gamma$ (see  \cite[Thm 1.1]{BJ}).  $\delta(\Gamma)$ also has an explicit relationship, via harmonic measure, to the lowest eigenvalue $\lambda_0(M)$ of the Laplacian (see \cite[Thm 4.2]{Corlette}). Having introduced the critical exponent, we can now state the main technical result of this paper.

\begin{thm}[{\bf The Jacobian bound}]
\label{MainBCGEstimate}
If $\Gamma$ is a discrete, finitely-generated, non-elementary, torsion-free subgroup of $\Isom^+(\Hy_\F^n)$ with $\Vol(\Hy_\F^n/\Gamma)=\infty$ and $F\colon \Hy_\F^n/\Gamma\to\Hy_\F^n/\Gamma$ is the barycenter map associated to the identity map, then for any $j\leq \min\{dn-3,d\}$ and $n>2$, the following holds:
\begin{equation}
\label{CFMJacEst}
\abs{\Jac_{dn-j}(F,x)} \leq \frac{2^{j/2} (\delta(\Gamma))^{d n-j}}{(d n-j-2)^j (d n+d-j-2)^{d n-2 j}}.
\end{equation}
For $n=2$, explicit bounds on $\Jac_{dn-j}(F,x)$ are given in \eqref{eq:ExceptionalCaseList} in the case $(d,j)$ is one of $(2,1)$, $(4,1)$, $(8,1)$, $(8,2)$, or $(8,3)$.
\end{thm}

Theorem~\ref{MainBCGEstimate} improves, for small $j$, known bounds for $\abs{\Jac_{dn-j}(F,x)}$ for barycenter maps $F$. The previously known bounds are not strong enough to deduce any case of Theorem~\ref{GenCDGapnew} (see Remarks \ref{remark:bcgnogood} and \ref{remark:nogood2} below). For a quaternionic-hyperbolic $n$--manifold $M$, Theorem~\ref{MainBCGEstimate} yields the estimate
\begin{equation}\label{eq:jacbound1}
\abs{\Jac_{4n-1}(F,x)} \leq \pr{\frac{\delta(\Gamma)}{4n+1}}^{4n-2}\pr{\frac{\sqrt{2}\delta(\Gamma)}{4n-3}} 
\end{equation}
for all $x$.  Corlette's Gap Theorem \cite[Thm 4.4]{Corlette} states that $\delta(\Gamma)\leq 4n$ when $\Gamma<\Sp(n,1)$ is not a lattice. Using this inequality in \eqref{eq:jacbound1} with some calculus gives the desired bound $ \abs{\Jac_{4n-1}(F,x)} \leq C_n < 1$.  The case $\Hy_\BB{O}^2$ is treated similarly.  

The subtlety of the upper bound $C_n$  (see \eqref{eq:SeqDef} for an explicit definition of $C_n$) given in \eqref{eq:jacbound1} can be seen in the following plot of $C_1,\ldots ,C_{34}$. In contrast to what often happens with such bounds, the quality of the bound barely improves as $n\to\infty$.  

\begin{minipage}{\linewidth}% to keep image and caption on one page
\makebox[\linewidth]{%        to center the image
  \includegraphics[keepaspectratio=true,scale=1]{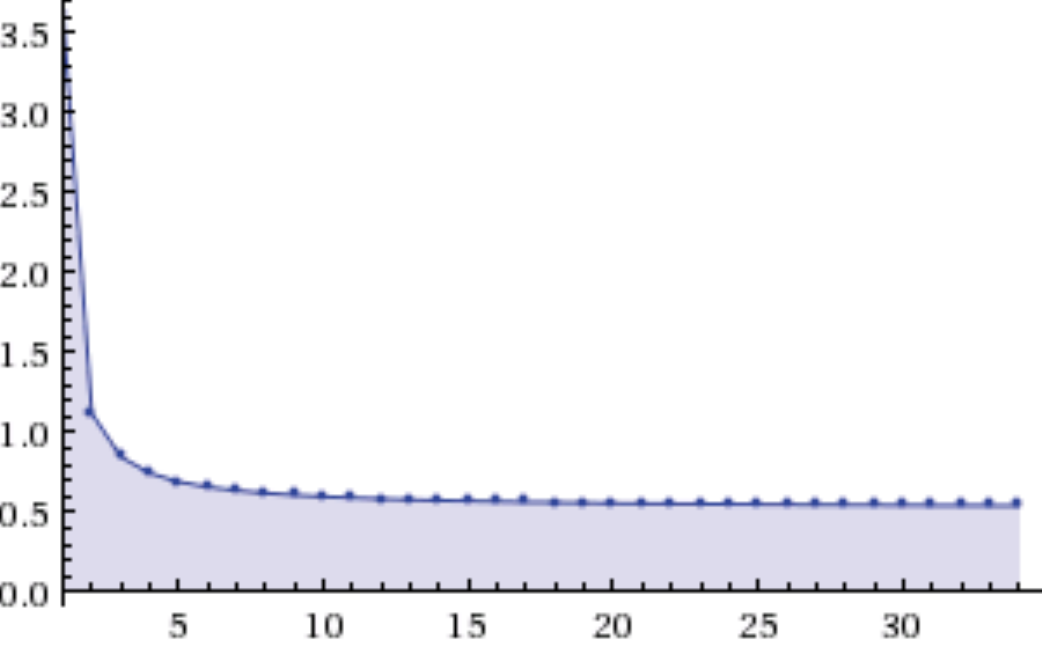}}
\captionof{figure}{The plot of the upper bound sequence $C_n$ derived from \eqref{eq:jacbound1}. Note that $C_n<1$ only for integers $n>1$, and that $C_n$ decreases extremely slowly as $n\to\infty$. }\label{fig:SeqPlot}  
\end{minipage}

\bigskip
As another indication of subtlety of the bound, note that if $\delta(\Gamma)$ were equal to $4n+1$, then our upper bound for $\abs{\Jac_{4n-1}(F,x)}$ would be greater than $1$. In particular, the explicit bounds in \cite{Corlette} (see Theorem~\ref{CorletteGap} below) are essential. 

The subtlety of this estimate manifests itself in its proof.  While the conceptual part of the proof of Theorem \ref{MainBCGEstimate} is given in \S\ref{JacEstSec}, it requires the solution of two elementary but quite involved optimization problems, which we black-box as Propositions \ref{prop:factor} and \ref{prop:Pest}. While the proofs there are self-contained, due to the computational nature of some parts we have also provided a Mathematica notebook to assist the reader (\cite{Mathematica}) where all explicit computations are carried out.

%---------------------------------------------------------------------------
\paragraph{Layout.} 

In \S\ref{JacEstSec}, we carry out the main part of the proof of Theorem \ref{MainBCGEstimate}. In \S\ref{S:Norms}, we briefly review norms on homology groups and Gromov's philosophy. In \S\ref{section:proofs1}, we prove Theorem \ref{GenCDGapnew}. In \S\ref{KapSection}, we apply Theorem~\ref{MainBCGEstimate} prove an inequality relating the critical exponent and homological dimension that generalizes an inequality of Kapovich. In \S\ref{section:optimizations}, we prove Proposition \ref{prop:factor} and Proposition \ref{prop:Pest}.

%---------------------------------------------------------------------------
%---------------------------------------------------------------------------
\paragraph{Acknowledgements.} The first author was partially supported by NSF and Simons Foundation grants. The second author was partial supported by NSF grants. The third author was partially supported by NSF grants DMS-1105710 and DMS-1408458. We would like to thank Gilles Carron, Kevin Corlette, Tom Farrell, Dave Futer, Artur Jackson, Pierre Py, Alan Reid, Juan Souto, Matthew Stover, Shmuel Weinberger, Sai-Kee Yeung, and Jiu-Kang Yu for extremely helpful discussions.

%---------------------------------------------------------------------------
%---------------------------------------------------------------------------
%%%%   JACOBIAN ESTIMATE       %%%%%%%%%%%%%
%---------------------------------------------------------------------------
%---------------------------------------------------------------------------
\section{The main Jacobian estimate}
\label{JacEstSec}

The goal of this section is to give a proof of Theorem \ref{MainBCGEstimate}, which is the main technical contribution of this paper.

%---------------------------------------------------------------------------
\subsection{The barycenter map}

In what follows we will adhere to the notation of Besson, Courtois, and Gallot \cite{BCGSchwartz}. For a discrete, finitely-generated, non-elementary, torsion-free subgroup $\Gamma$ of $G$, in all that follows, we will denote the associated $\F$--hyperbolic $n$--manifold by $M = \Hy^n_\F/\Gamma$. For any continuous map $f\colon M\to N$, where $N$ is a complete $\F$--hyperbolic $n$--manifold and $f$ has non-zero degree, Besson--Courtois--Gallot \cite{BCGRigid, BCGSchwartz} constructed a remarkable map $F\colon M\to N$ homotopic to $f$, called the \textbf{barycenter map}. We briefly review the construction of $F$ in a bit more detail than given in the introduction. In what follows, we use standard results about Patterson--Sullivan measures and refer the reader to Bishop--Jones \cite{BJ} or Sullivan \cite{Sullivan} (see also \cite{BCGRigid, BCGSchwartz}). 

Let $\Cal{M}_1(\partial \Hy_\F^n)$ denote the space of probability measures on 
$\partial \Hy_\F^n$.  Consider the linear functional $\Cal{B}\colon \Cal{M}_1(\partial \Hy_\F^n)\times \Hy_\F^n \to \R$ given by 
\[ \Cal{B}(\nu,x)=\int_{\partial \Hy_\F^n} B_\xi(p,x)d\nu(\xi), \] 
where $B_{\xi}(p,x)$ is the Busemann function associated to the boundary point $\xi$ with base point $p$. For any fixed measure $\nu$ that is not the sum of two dirac measures, this integral of convex functions is itself strictly convex. Therefore $\Cal{B}(\nu,\cdot)$ has a unique minimum on $\Hy_\F^n\cup \partial \Hy_\F^n$ which is its unique critical point, denoted here by $\bary(\nu)$.

As mentioned in the introduction, any continuous map $f\colon M \to N$ of nonzero degree induces a map $\pr{\partial \widetilde{f}}_\ast\colon \Cal{M}_1(\partial \Hy_\F^n) \to \Cal{M}_1(\partial \Hy_\F^n)$. We set $\widetilde{F}(x)=\bary\pr{\pr{\partial\widetilde{f}}_\ast(\mu_x)}$, where $\mu_x$ is the Patterson--Sullivan measure at $x$ on $\partial \Hy_\F^n$ associated to $\Gamma$. Note that $\widetilde{F}$ is well-defined since $\mu_x$ has full support and is never a pair of atoms, since $\Gamma$ is non-elementary. Changing the basepoint $p\in \Hy_\F^n$ changes each Busemann function by a constant, and therefore does not affect the location of the minimum of $\Cal{B}(\nu,\cdot)$ or the map $\widetilde{F}$. Whenever $\widetilde{f}$ is equivariant with respect to an isometry $\ga$, i.e. $\ga \widetilde{f}(x)=\widetilde{f}(\ga x)$, then the map $\widetilde{F}$ is also equivariant with respect to $\ga$ since $\mu_{\ga x}=\ga_\ast \mu_x$ and Busemann functions are equivariant. Hence $\widetilde{F}$ descends to give the barycenter map $F\colon M\to N$ of Besson--Courtois--Gallot.

%---------------------------------------------------------------------------
%---------------------------------------------------------------------------
\subsection{Estimating the $p$--Jacobian}

In order to better keep track of the domain and codomain, we let $N$ be another copy of $M$. We begin by letting $F\colon M \to N$ be the barycenter map associated to the identity map $f\colon M\to N$. This assumption on the map is not necessary for the estimates that follow, although this case is all that is needed for our applications. 

\begin{proof}[Proof of Theorem \ref{MainBCGEstimate}]
We establish the inequality in four main steps.

%---------------------------------------------------------------------------
\paragraph{Step 1 (Setup): }

We begin with some setup.  In the definition of $p$--Jacobian \eqref{eq:JacDef}, as the space of $p$--dimensional subspaces of $T_xM$ is compact, there is a subspace $U_x \su T_x M$ such that $\Jac_{p}(F,x)$ is maximized at $U_x$ with $V_x = dF_x(U_x)$. For our estimate of $\abs{\Jac_{p}(F,x)}$, we have 
three maps defined by Formula 2.4 in  \cite{BCGSchwartz}: 
\begin{align*}
h_x'\colon & T_x M \lra T_x M, & h_x\colon & T_{F(x)} N \lra T_{F(x)} N, & k_x\colon & T_{F(x)} N \lra T_{F(x)} N.
\end{align*}
In our case these simplify to the following:
\begin{align*}
h_x'(u)&=\int_{\partial \Hy_\F^n} v_{\tilde x,\xi}\innp{v_{\tilde x,\xi},u} d\mu_{\tilde x}(\xi) \\
h_x(v)&=\int_{\partial \Hy_\F^n} v_{\tilde{F}(\tilde x),\xi}\innp{v_{\tilde{F}(\tilde x),\xi},v} d\mu_{\tilde x}(\xi)\\
k_x(v)&=\int_{\partial \Hy_\F^n} Dd_{\tilde{F}(\tilde x)}B_{\xi}(v) d\mu_{\tilde x}(\xi),
\end{align*}
where $\tilde{x}$ is a chosen lift of $x$, and $v_{\tilde x,\xi}$ is the unique unit tangent vector in $S_{\tilde{x}} \Hy_\F^n$ tangent to the geodesic ray from $\tilde{x}$ to $\xi$.  The second covariant derivative of $B_\xi$ can be expressed in terms of the first covariant derivative of the gradient $Dd_{\tilde{F}(\tilde x)}B_{\xi}(v)=\nabla_v \nabla_{\tilde{F}(\tilde x)}B_{\xi}(p,\cdot)$. We have identified the tangent spaces $T_xM$ in the base with those of the universal cover $T_{\tilde x}\Hy_\F^n$. Note that $h_x$ and $h_x'$ are identical, except that they are interpreted to be on the codomain and domain spaces, respectively.

%---------------------------------------------------------------------------
\paragraph{Step 2 (Estimates from Besson--Courtois--Gallot \cite{BCGSchwartz}):}  

For any subspace $W$ of $V$ and any positive definite, symmetric, linear operator $A\colon V \to V$, the linear operator $A_W:= \mathrm{Proj}_W \circ A_{|W}$ is also symmetric and positive semidefinite. Moreover, if $\innp{\cdot,\cdot}_A$ is the bilinear form on $V$ given by $\innp{v,w}_A = \innp{Av,w}$ and $\innp{\cdot,\cdot}_{A,W}$ is the restriction of $\innp{\cdot,\cdot}_A$ to $W \times W$, an elementary argument shows that $\innp{\cdot,\cdot}_{A,W}$ is the bilinear form associated to $A_W$. Applying this restriction to the linear maps above gives symmetric linear maps
\begin{align*}
k_{x,V}\colon & V \lra V, & h_{x,V}\colon & V \lra V, & h_{x,U}'\colon & U \lra U. 
\end{align*}
Following \cite{BCGSchwartz}, we compute $\det(k_{x,V} \circ dF_x)$ with respect to orthonormal bases on $U,V$, respectively. Formula 2.6 in \cite{BCGSchwartz} states the following inequalities: 
\begin{equation}
\label{eq:det1}
\begin{array}{ll}
 \det k_{x,V} \cdot \abs{\Jac_{p}(F,x)}& \leq (\delta(\Gamma))^p(\det h_{x,V})^{1/2}(\det h_{x,U}')^{1/2}\\
 & \\
 & \leq  (\delta(\Gamma))^p(\det h_{x,V})^{1/2} \cdot \pr{\frac{1}{p}\textrm{Tr}(h'_{x,U})}^{p/2}
 \end{array}
 \end{equation}
Formula 2.7 in \cite{BCGSchwartz} states that $\textrm{Tr}(h'_{x,U})\leq 1$. Plugging this inequality into the inequality \eqref{eq:det1} yields 
 \begin{equation}
 \label{eq:jac:estimate}
 \abs{\Jac_{p}(F,x)} \leq \frac{(\delta(\Gamma))^p(\det h_{x,V})^{1/2}}{p^{p/2}\det k_{x,V}}.
 \end{equation}
As our aim is to bound the left-hand side of \eqref{eq:jac:estimate}, we require bounds on each of the terms on the right-hand side.

%---------------------------------------------------------------------------
\paragraph{Step 3 (An estimate from linear algebra):  }

The natural involutions on the algebra $\F$ induce involutions on the tangent spaces and hence tangent bundles of $\Hy^n_\F$.  We will refer to these maps as the \textbf{$\F$--structure maps} and will denote them by $\tau_1,\dots,\tau_{d-1}$. When $\F=\BB{H}$, for example, the three structure maps $\tau_1, \tau_2, \tau_3$ corresponding to multiplication by  $\bf{i}, \bf{j}, \bf{k}$. Each $\tau_i$ is an isometry of the Riemannian metric at each $x\in \Hy^n_\F$ and commutes with the isotropy subgroup at $x$. It follows (see, e.g.\ \cite{BCGSchwartz}, top of p.~155) that
\[ k_{x} = \Id - h_{x} - \sum_{i=1}^{d-1} \tau_i h_{x} \tau_i.\]
Given an arbitrary subspace $W$ of $V$, we have
\begin{equation}
\label{eq:kxw}
k_{x,W}= \Id_W - h_{x,W} + \sum_{i=1}^{d-1} (-\tau_i h_{x} \tau_i)_W.
\end{equation}
Since $h_x$ is symmetric, positive definite with eigenvalues strictly less than $1$, it follows that for each $i \geq 1$, the map $-\tau_i h_{x} \tau_i$ is symmetric and positive definite. In particular, so are each  of the mappings $-(\tau_i h_{x} \tau_i)_W$. Even though $W$ is not $\tau_i$--invariant in general, we can still analyze the situation using work of Fiedler \cite{Fiedler}. We refer the reader to the remark starting on p.~29 of \cite{Fiedler} for the justification of the following lemma.

\begin{lemma}[Fiedler \cite{Fiedler}]
\label{lemma:Fiedler}
If $A_0,\dots,A_k \in \Mat(n,\C)$ are positive semidefinite, Hermitian matrices such that the eigenvalues of $A_j$ are given by $0\leq \alpha_{1,j} \leq \cdots \leq \alpha_{n,j}$, then 
\[ \det\pr{\sum_{j=0}^k A_j} \geq \prod_{i=1}^{n}\pr{\sum_{j=0}^k\alpha_{i,j}}.\] 
\end{lemma}

In order to apply Lemma~\ref{lemma:Fiedler} to Equation~\eqref{eq:kxw}, we need some additional notation. Let 
$1\geq \la_1\geq\cdots\geq \la_p\geq 0$ be the eigenvalues of $h_{x,W}$ with $p =\dim W$, and $\beta_1\leq\cdots\leq \beta_{nd}$ be the eigenvalues of $h_{x}$ on all of $T_xN$; these are also the eigenvalues of each $-\tau_i h_{x} \tau_i$.   The Courant--Fischer Min-Max Theorem implies 
\begin{align}
\label{eq:eig-ineq}
\la_i\leq \beta_{nd-i+1}
\end{align}
for all $1\leq i \leq d-1$.  Since $h_{x}$ is an integral of symmetric, positive semidefinite linear maps whose eigenvalues sum to $1$, the map $h_x$ inherits these properties as well. Moreover, since the measure is not single atom, the eigenvalues are strictly less than $1$. In particular, $\Id_W - h_{x,W}$ must be positive definite Hermitian.    

We can now apply Lemma~\ref{lemma:Fiedler} with $A_0=\Id_W - h_{x,W}$ and $A_j=(-\tau_{j} h_{x} \tau_{j})_W$ for each $j=1,\dots,d-1$. Note that by \eqref{eq:kxw} we have $k_{x,W}= \sum_{j=0}^{d-1} A_j$. Since the $i$th smallest eigenvalue of $A_0$ is $1-\la_i$ and the $i$th smallest eigenvalue of each $A_j$ is $\beta_i$ for $j>0$, an application of Lemma~\ref{lemma:Fiedler} gives the inequality 
\begin{equation}
\label{eq:detkxw}
\det k_{x,W}\geq \prod_{i=1}^p(1-\la_{i}+(d-1)\beta_i).
\end{equation}

%---------------------------------------------------------------------------
\paragraph{Step 4 (Reducing to two optimization problems): }

We now consider the problem of minimizing the right-hand side of \eqref{eq:detkxw} as a function of the $\beta_i$, subject to the constraint given by the inequality \eqref{eq:eig-ineq}. First, we view \eqref{eq:detkxw} as a product of the form $\prod_{i=1}^p (a_i+b_i)$ with $1\geq a_p \geq \dots \geq a_1\geq 0$ fixed and $b_i$ chosen from a fixed set of positive numbers. We provide an overestimation by pairing the smallest of the $a_i$ with the smallest of the $b_j$. It is straightforward to verify that the minimum occurs when the smallest among the $b_i$ is matched with the smallest among the $a_i$ and so on. In particular, the minimum of the right-hand side of \eqref{eq:detkxw} occurs when we choose the smallest possible values for $1-\la_i$, namely $\la_i=\beta_{nd-i+1}$ and match these with $\beta_i$ for $1\leq i\leq p$. As the first $nd-p$ of the $\beta_i$ will not coincide with any $\la_j$, the minimum will occur when $\beta_i=0$ for $1\leq i\leq nd-p$. Writing the other $\beta_i$ in terms of $\la_i$, we have $\beta_i=\la_{nd+1-i}$ for $nd-p+1\leq i\leq nd$ and $\beta_i=0$ for $i=1,\dots,j$. This case corresponds to case when $W$ contains the eigenspaces of $h_{x}$ corresponding to the $p$ highest eigenvalues and $\tau_i$ exchanges the top $nd-p$ eigenspaces for $W^{\perp}$, which happens to be the eigenspace for the lowest possible eigenvalue $0$. In particular, setting $W=V$ and $j=nd-p$ gives 
\begin{equation}
\label{eq:eigs}
 \frac{\sqrt{\det h_{x,V}}}{\det k_{x,V}}\leq \frac{\sqrt{\la_1\dots \la_p}}{\prod_{i=j+1}^{p}(1-\la_i+(d-1)\la_{nd+1-i})\prod_{i=1}^j (1-\la_{i})}.
\end{equation}

The next proposition will be our main tool for bounding the right hand side of \eqref{eq:eigs}.

\begin{prop}[{\bf First optimization}]
\label{prop:factor}
If $n\geq 2$, or $n>2$ if $d=1$, then the maximum of 
\[
f(x_1,\dots,x_p)=\frac{x_1\dots x_p}{\prod_{i=j+1}^{p}(1-x_i+(d-1)x_{nd+1-i})^2\prod_{i=1}^j (1-x_{i})^2}
\]
subject to $x_1+\dots+x_p\leq 1$ is achieved at a point where
\[ x_1=x_2=\dots=x_j=\sigma, \quad x_{j+1}=x_{j+1}=\dots=x_{p}=\la \] 
with $(p-j)\la+j \sigma=1$.
\end{prop}

We have relegated the proof of Proposition \ref{prop:factor} to \S \ref{AppA} since it only involves elementary methods. Note that $\sigma$ and $\lambda$ must be nonnegative due to the constraint functions being a sum of nonnegative $\la_i$. Applying Proposition \ref{prop:factor}, the maximum for the right hand side of \eqref{eq:eigs} occurs when $\la=\la_{j+1}=\cdots=\la_{p}$, $\sigma=\la_{1}=\cdots=\la_j$, and 
\begin{equation}
\label{eq:lamandsig}
(p-j)\la+j\sigma=1. 
\end{equation}
%Note that at the maximum, we must have $\sigma \geq \lambda$, with equality only when $d=1$, since otherwise the denominator of the left hand side would give a greater contribution than the right hand denominator per eigenvalue thus lowering the potential maximum. 

It now remains to find or estimate the values of $\la$ and $\sigma$ at the maximum. Inserting $\la,\sigma$ into \eqref{eq:eigs} gives 
\begin{equation}
\label{eq:JacEst}
\frac{\sqrt{\det h_{x,W}}}{\det k_{x,W}}\leq \frac{\lambda^{(p-j)/2}\sigma^{j/2}}{(1+(d-2)\la)^{p-j}(1-\sigma)^j}.
\end{equation}
We are trying to bound from above the right-hand side of this inequality. We denote this quantity by $P(\la,\sigma)$. In the case that $j=0$ we have the former solution $\la=\frac1p$ for the maximum. Hence we will assume that $j\geq 1$, and since we assumed $j<d$, we also have $d\geq 2$. Together with the assumption that $n\geq 2$, we obtain
\[ p=nd-j\geq 2(j+1)-j\geq 3. \]
Solving for $\sigma$ in terms of $\lambda$ in \eqref{eq:lamandsig} and inserting the result into the right-hand side of \eqref{eq:JacEst}, yields an expression with $\lambda$ as the only variable. We denote the results of this endeavor by $P(\lambda)=P(\la,\frac{1+(p-j)\la}{j})$. 

\begin{prop}[{\bf Second optimization}]
\label{prop:Pest}
Whenever
\begin{align}\tag{$\dagger$}
\label{eq:hypoth}
0\leq \la\leq 1, \ \ d=2,4,8, \ \  0<j<d,\ \  p=nd-j, \ \ \text{and}\ \  n=2 \text{ when }d=8,
\end{align}
the function
\begin{align}\label{eq:Pfunc}
P(\la)=\left[\frac{j^j \la^{p-j} \left(1-(p-j)\la\right)^j}{(j-1+\la (p-j))^{2 j} (1+(d-2) \la)^{2 (p-j)}}\right]^{1/2}
\end{align}
is bounded by
\[ P(\la) \leq \frac{2^{j/2} p^{p/2}}{(p-2)^j (p+d-2)^{p- j}} \]
except when
\begin{align}\tag{$\ddagger$}\label{eq:hypoth2}
n=2 \ \ \text{and}\ \ j=1,\ \  \text{or}\ \ n=2,\ \  j=2,3\ \ \text{and}\ \ d=8.
\end{align}
We have bounds for the $n=2$, $j=1$ and $d=2,4,8$ cases
\[ P(\la)\leq \frac{1}{2}, \quad P(\la)\leq \frac{(3)^5 (\sqrt{13})(\sqrt{37})}{(2)^3 (11)^6}, \quad P(\la)\leq \frac{(2)^{13} (3)^{7}\sqrt{3} (7)^6 (29)^7 (\sqrt{17})}{(5)^{28} (11)^{14}}, \]
and for the $n=2$, $d=8$ and $j=2,3$ cases
\[ P(\la)\leq \frac{(3)^6 (5)^{13} }{(2)^{22}(17)^{12}}, \quad P(\la)\leq \frac{(6\sqrt{6}) (5)^{12}(7) ^5}{(167)^{10}}. \]
\end{prop}

As the proof of Proposition \ref{prop:Pest} involves only elementary mathematics, it has been moved to \S \ref{AppB}. The exceptional cases of the Jacobian estimates are given in \eqref{eq:ExceptionalCaseList}. 

The proof of the theorem is completed by multiplying the estimates in Proposition \ref{prop:Pest} by $\frac{(\delta(\Ga))^p}{p^{p/2}}$.
\end{proof}

\begin{rem}
\label{remark:bcgnogood}
In \cite[Thm 1.10 (ii)]{BCGSchwartz}, Besson--Courtois--Gallot  treat the case $\F=\R$.  For general $\F$--hyperbolic $n$--manifolds, that method gives the inequality
\begin{equation}\label{GenBCG}
\abs{\Jac_{dn-j}(F,x)} \leq \pr{\frac{\delta(\Gamma)}{dn-j-1}}^{dn-j}.
\end{equation}
The upper bound for $\Jac_{dn-j}(F,x)$ obtained from combining \eqref{GenBCG} with Theorem~\ref{CorletteGap} is insufficient for our purposes: it does not yield \eqref{equation:shrink} needed to prove Theorem~\ref{GenCDGapnew}.   We obtain the stronger Jacobian bound of Theorem~\ref{MainBCGEstimate} by exploiting the directions of negative curvature $-4$ that exist when $\F\neq\R$. The inequality \eqref{GenBCG} only uses that sectional curvatures are bounded above by $-1$.
\end{rem}

\begin{rem}
Recently, Kim--Kim \cite{Kim2} and Lafont--Wang \cite{LafontWang} gave general estimates for $p$--Jacobians in all real ranks. Those bounds do not suffice for our applications.
\end{rem}

%---------------------------------------------------------------------------
%--------------------------------------------------------------------------
%%%%%     HOMOLOGY NORMS     %%%%%%%%%%%%
%---------------------------------------------------------------------------
%---------------------------------------------------------------------------
\section{The mass of cycles}\label{S:Norms}

In order to use Theorem \ref{MainBCGEstimate} to prove Theorem \ref{GenCDGapnew}, we need to relate volume to homology. This topic can be found in Gromov \cite[Ch.~4-5]{Gromov}. 

%---------------------------------------------------------------------------
%---------------------------------------------------------------------------
\subsection{Mass with coefficients}
\label{subsection:mass}

Throughout, we set $M = \Hy_\F^n/\Gamma$ for a discrete, finitely-generated, torsion-free subgroup $\Gamma$ of $\Isom^+(\Hy_\F^n)$. The homology groups $H_p(M;\R)$ can be equipped with norms in several ways. The norms we use are those utilized by Kapovich \cite[\S3]{Kapovich} and are derived from a notion of mass used by Gromov \cite[4.15]{Gromov} (see also \cite{GBC}). The \textbf{co-mass} of a $p$--form $\omega$ is given by
\[ \norm{\omega}_{co-mass} = \sup_{v_1,\dots,v_p} \abs{\omega(v_1,\dots,v_p)}, \]
where $v_1,\dots,v_p$ ranges over all orthonormal $p$--frames.   One defines the mass first on the class of Lipschitz cycles $c\colon X_p \to M$, where $X_p$ is the standard Euclidean $p$--simplex. For such a map $c$, the \textbf{mass} is defined to be
\[ \norm{c}_{mass} := \sup\set{\abs{\int_c \omega}~:~ \omega \in \Lambda^p(M),~\norm{\omega}_{co-mass} \leq 1}. \]
For a Lipschitz $p$--chain $c$ given by $c = \sum_{j=1}^r \beta_j c_j$ define
\[ \norm{c}_{mass} := \sum_{j=1}^r \abs{\beta_j} \norm{c_j}_{mass}. \]
Taking the infimum over all representatives then induces a norm on the homology groups: for $\xi \in H_p(M;\R)$, define the \textbf{mass (or volume) of the homology class $\xi$} by 
\[ \norm{\xi}_{mass} = \inf\set{\norm{c}_{mass}~:~ [c] = \xi}. \]

Kapovich defines these norms in \cite[\S~3]{Kapovich} for homology $H_p(M;V)$ with coefficients in any flat bundle $B_M$ associated to any $\Gamma$--module $V$.  As noted in \cite[p.~2028]{Kapovich},  the generalized coefficients make no contribution to these norms.  

%---------------------------------------------------------------------------
%---------------------------------------------------------------------------
\subsection{A vanishing theorem}

The main homological vanishing result that we require can now be stated.

\begin{thm}[\cite{Gromov}, \cite{Kapovich}]\label{T71}
If $\Gamma$ is a discrete, finitely-generated, torsion-free subgroup of $\Isom^+(\Hy_\F^n)$ with no parabolic elements, $M$ is the associated manifold for $\Gamma$, and $V$ is a $\Gamma$--module associated to a flat bundle over $M$, then there exists a positive constant $\theta_{\F,n}$ such that every homology class $\xi \in H_p(M;V)$ with mass less than $\theta_{\F,n}$ is trivial.
\end{thm}

Theorem~\ref{T71} was established by Kapovich \cite[Thm 7.1]{Kapovich} for real-hyperbolic $n$--manifolds. As noted in Remark 6.8 of \cite{Kapovich}, it is straightforward to extend the proof to the present setting. For our purposes, we require the following corollary.

\begin{cor}\label{SmallJacCorollary}
Let $\Gamma$ be a discrete, finitely-generated, torsion-free subgroup of $\Isom^+(\Hy_\F^n)$ with no parabolic elements, $M$ is the associated manifold for $\Gamma$, and $F\colon M \to M$ a smooth map that is homotopic to the identity map. If $\abs{\Jac_{p}(F,x)} \leq C < 1$ for some $p \leq dn$ and all $x \in M$, then $H_p(M;V)=0$ for any $\Gamma$--module $V$.
\end{cor}

Corollary~\ref{SmallJacCorollary} is obtained directly from Theorem \ref{T71} by applying $F$ a sufficient number of times to any representative $c$ of a class $\xi \in H_p(M;V)$ in order to push the mass below the threshold $\theta_{\F,n}$. In fact, we only require that nontrivial classes have positive mass. 

%---------------------------------------------------------------------------
%---------------------------------------------------------------------------
%%%%    PROOF OF VANISHING    %%%%%%%%%%%%%
%---------------------------------------------------------------------------
%---------------------------------------------------------------------------
\section{Proof of Theorem \ref{GenCDGapnew}}
\label{section:proofs1}

Before proving Theorem \ref{GenCDGapnew}, we state Corlette's Gap Theorem \cite[Thm 4.4]{Corlette}.

\begin{thm}[\cite{Corlette}]\label{CorletteGap}
Let $\Gamma<\Isom(\Hy^n_\F)$ be a discrete subgroup.
\begin{itemize}
\item[(a)] 
If $\F=\BB{H}$, then either $\delta(\Gamma) = 4n+2$ or $\delta(\Gamma) \leq 4n$. The former happens if and only if $\Gamma$ is a lattice in $\Isom(\Hy_\BB{H}^n)$. 
\item[(b)] 
If $\F=\BB{O}$, then either $\delta(\Gamma) = 22$ or $\delta(\Gamma) \leq 16$. The former happens if and only if $\Gamma$ is a lattice in $\Isom(\Hy_\BB{O}^2)$.
\end{itemize}
\end{thm}

We can now prove Theorem \ref{GenCDGapnew}.

\begin{proof}[Proof of Theorem \ref{GenCDGapnew}]
Let $M:=(\Hy^n_{\F}/\Gamma)$ be the $\F$--hyperbolic manifold associated to $\Gamma$, let $B$ be the associated flat bundle for a fixed $\Gamma$--module $V$, and $F\colon M \to M$ the barycenter map associated to the identity. Recall that we are assuming that $\Vol(M)=\infty$ and that $\Gamma$ has no parabolic elements. 

\paragraph{The case \boldmath$\Gamma<\Isom^+(\Hy^n_{\BB{H}}), n\geq 2$.}  Our goal is to establish that $H_{4n-1}(M;V)$ is trivial. To begin, Theorem~\ref{MainBCGEstimate} gives, for $n>2$, the inequality:
\begin{equation}
\label{eq:Jac5}
\abs{\Jac_{4n-1}(F,x)}  \leq \pr{\frac{\delta(\Gamma)}{4n+1}}^{4n-2}\pr{\frac{\sqrt{2}\delta(\Gamma)}{4n-3}}. 
\end{equation}
Since $\Vol(M)=\infty$, Theorem \ref{CorletteGap} gives that $\delta(\Gamma) \leq 4n$. Using this inequality in \eqref{eq:Jac5} gives 
\begin{equation}
\label{eq:Jac55}
\abs{\Jac_{4n-1}(F,x)} \leq \pr{\frac{4n}{4n+1}}^{4n-2}\pr{\frac{\sqrt{2}\cdot 4n}{4n-3}}.
\end{equation}
Taking the limit of the right hand side of \eqref{eq:Jac55} as $n\to\infty$, we obtain: 
\[ \lim_{n \to \iny} \pr{\frac{4n}{4n+1}}^{4n-2}\pr{\frac{\sqrt{2}\cdot 4n}{4n-3}}=\frac{\sqrt{2}}{e} \approx 0.52026009502 <1.\]
The sequence
\begin{equation}
\label{eq:SeqDef}
C_n:=\pr{\frac{4n}{4n+1}}^{4n-2}\pr{\frac{\sqrt{2}\cdot 4n}{4n-3}}
\end{equation}
is strictly decreasing for all $n$ as the derivative of $\log C_n$ is negative for all $n \geq 1$.  At $n=3$, the sequence $C_n$ takes the value
\[C_3=\pr{\frac{12}{13}}^{10}\pr{\frac{\sqrt{2}\cdot 12}{9}} \approx 0.84690105104 < 1.\]
It follows that $\abs{\Jac_{4n-1}(F,x)}<1$ for $n\geq 3$. In the exceptional case $n=2$, we have, after an application of Theorem~\ref{CorletteGap}, the inequality 
\[ \abs{\Jac_{4\cdot 2 - 1}(F,x)} = \frac{(3)^5 (\sqrt{13})(\sqrt{37})(\delta(\Gamma))^7}{(\sqrt{7}) (2)^3 (7)^3 (11)^6}\leq \frac{(3)^5 (\sqrt{13})(\sqrt{37})(8)^7}{(\sqrt{7}) (2)^3 (7)^3 (11)^6} \approx 0.8689994123 <1.\]
In particular $\abs{\Jac_{4n-1}(F,x)}<1$ for $n\geq 2$. Corollary~\ref{SmallJacCorollary} then implies $H_{4n-1}(M;V) = 0$.

\paragraph{The case \boldmath$\Gamma<\Isom^+(\Hy^n_{\BB{O}})$.}  We must prove that $H_j(M;V) = 0$ for any finitely-generated $\Gamma$--module $V$ and $j=13,14,15$.  Since $\Vol(M)=\infty$, Theorem~\ref{CorletteGap} implies that $\delta(\Gamma) \leq 16$. Using this inequality in Theorem~\ref{MainBCGEstimate} gives three inequalities: 
\begin{align*}
\abs{\Jac_{8\cdot 2 - 1}(F,x)}  &\leq \frac{(2)^{13} (7)^6 (29)^7 (\sqrt{17})(16)^{15}}{(\sqrt{5}) (5)^{35} (11)^{14}} \approx 0.03197831847 < 1 \\
\abs{\Jac_{8\cdot 2 - 2}(F,x)}  &\leq  \frac{(3)^6 (5)^{13} (16)^{14}}{(2)^{29}(7)^7(17)^{12}} \approx 0.24892821847 < 1 \\
\abs{\Jac_{8\cdot 2 - 3}(F,x)}  &\leq  \frac{(6\sqrt{6}) (5)^{12}(7) ^5(16)^{13}}{\sqrt{13} (13)^{6}(167)^{10}} \approx 0.92495456626 < 1.
\end{align*}
Corollary \ref{SmallJacCorollary} then implies that $H_j(M;V)=0$ for $j=13,14,15$.
\end{proof}

\begin{rem}
\label{remark:nogood2}

For $n>2$ and $j=2$, our estimate for $\abs{\Jac_{4n-2}(F,x)}$ in combination with Theorem \ref{CorletteGap} yields
\begin{equation}
\label{eq:Jac555}
\abs{\Jac_{4n-2}(F,x)} \leq \frac{2(4n)^{4n-2}}{(4n-4)^2(4n)^{4n-4}}.
\end{equation}
The right-hand side of \eqref{eq:Jac555} is greater than $2$ and so insufficient for proving the vanishing of $H_{4n-2}(M;V)$. Note that Li \cite[Cor.~6.5]{Li} proved that $H_{4n-2}(M;\C)$ is virtually nontrivial for any closed quaternionic-hyperbolic $n$--manifold $M$. The best estimate we obtain for $\abs{\Jac_{8\cdot 2 - 4}(F,x)}$ using our methods gives an upper bound larger than 1. It is unknown if $H_{12}(M;\Q)$ is virtually nontrivial for a closed Cayley-hyperbolic $2$--manifold. 
\end{rem}

%---------------------------------------------------------------------------
%---------------------------------------------------------------------------
%%%%    KAPOVICH INEQUALITY    %%%%%%%%%%%%
%---------------------------------------------------------------------------
%---------------------------------------------------------------------------
\section{Critical exponent versus homological dimension}
\label{KapSection}

One of the main results of Kapovich \cite[Thm 1.1]{Kapovich} is the inequality
\begin{equation}\label{KapHom}
\delta(\Gamma) \geq \hd(\Gamma) - 1,
\end{equation}
where $\Gamma < \Isom^+(\Hy_\R^n)$ contains no parabolic elements and $\hd(\Gamma)$ is the homological dimension of $\Gamma$. Our next application of Theorem~\ref{MainBCGEstimate} is the following inequality.

\begin{thm}\label{MainXKapovich}
Let $\Gamma < \Isom^+(\Hy_\F^n)$ be a discrete, finitely-generated, torsion-free subgroup with $\Vol(\Hy^n_\F/\Gamma)=\infty$ and no parabolic elements. If $\hd(\Gamma)> dn - d$ with $n>2$, then
\begin{equation}\label{CFMHom}
\delta(\Gamma) \geq \pr{\frac{\hd(\Gamma)-2}{\sqrt{2}}}^{\displaystyle \pr{\frac{dn}{\hd(\Gamma)} - 1}}\pr{\hd(\Gamma)-2+d}^{^{\displaystyle \pr{2-\frac{dn}{\hd(\Gamma)}}}}.
\end{equation}
\end{thm}

\begin{proof}
We assume that $\hd(\Gamma) > dn - d$ and set $\hd(\Gamma) = dn - j$. In particular, Theorem~\ref{MainBCGEstimate} is applicable and we obtain
\begin{equation}\label{LocalKapEQ}
\abs{\Jac_{\hd(\Gamma)}(F,x)} \leq \frac{2^{j/2} (\delta(\Gamma))^{\hd(\Gamma)}}{(\hd(\Gamma)-2)^j (\hd(\Gamma)-j-2)^{d n-2 j}}.
\end{equation}
Inserting $j = dn - \hd(\Gamma)$ into \eqref{LocalKapEQ} yields 
\[ \abs{\Jac_{\hd(\Gamma)}(F,x)} \leq \frac{\sqrt{2}^{dn - \hd(\Gamma)}(\delta(\Gamma))^{\hd(\Gamma)}}{(\hd(\Gamma) - 2)^{dn - \hd(\Gamma)}(\hd(\Gamma)-2+d)^{2\hd(\Gamma)-dn}}. \]
Corollary~\ref{SmallJacCorollary} then implies $1 \leq \abs{\Jac_{\hd(\Gamma)}(F,x)}$ and so
\[ 1 \leq \frac{\sqrt{2}^{dn - \hd(\Gamma)}(\delta(\Gamma))^{\hd(\Gamma)}}{(\hd(\Gamma) - 2)^{dn - \hd(\Gamma)}(\hd(\Gamma)-2+d)^{2\hd(\Gamma)-dn}}. \]
Therefore
\[ \frac{(\hd(\Gamma)-2)^{dn - \hd(\Gamma)}(\hd(\Gamma)-2+d)^{2\hd(\Gamma)-dn}}{\sqrt{2}^{dn - \hd(\Gamma)}} \leq (\delta(\Gamma))^{\hd(\Gamma)}, \]
and consequently
\[ \pr{\frac{(\hd(\Gamma)-2)}{\sqrt{2}}}^{\displaystyle \pr{\frac{dn}{\hd(\Gamma)} - 1}}(\hd(\Gamma)-2+d)^{^{\displaystyle \pr{2-\frac{dn}{\hd(\Gamma)}}}} \leq \delta(\Gamma). \]
\end{proof}

Although Kapovich stated \eqref{KapHom} only for the case $\F=\R$, his proof can be extended to 
$\F=\C, \BB{H}$ and $\BB{O}$.  Our inequality \eqref{CFMHom} is more complicated and also requires $\hd(\Gamma)$ to be sufficiently large. However, combining \eqref{CFMHom} with some elementary estimates, we see that for any $\eps >0$ there exists $n_\eps \in \N$ such that if $n \geq n_\eps$ and $\Gamma < \Isom^+(\Hy_\F^n)$ satisfies the assumptions of Theorem \ref{MainXKapovich}, then
\begin{equation}\label{SimpleCFMHom}
\delta(\Gamma) \geq \hd(\Gamma) - 2 + d - \eps.
\end{equation}
In particular, when $d>1$ and $n$ is sufficiently large, (\ref{CFMHom}) gives a strictly better lower bound for $\delta(\Gamma)$ than \eqref{KapHom}. 

%---------------------------------------------------------------------------
%---------------------------------------------------------------------------
%---------------------------------------------------------------------------
%---------------------------------------------------------------------------
%%%%%%%%% APPENDIX  A%%%%%%%%%%%%%%%%
%---------------------------------------------------------------------------
%---------------------------------------------------------------------------
%---------------------------------------------------------------------------
%---------------------------------------------------------------------------

\section{Solving the two optimizations}
\label{section:optimizations}

In this section we prove Proposition \ref{prop:factor} and Proposition \ref{prop:Pest}. The proofs are self-contained, but due to the computational nature of some parts we have constructed a publicly available Mathematica notebook (\cite{Mathematica}) where all explicit computations are carried out in complete detail.

\subsection{Proof of  Proposition \ref{prop:factor}}
\label{AppendixSectionA}\label{AppA}

To prove Proposition \ref{prop:factor}, we require a pair of lemmas.

\begin{lemma}\label{lem:split_opt}
Any maximum point for an objective function of the form 
\[ f=f_1(x_1,\dots,x_k)f_2(x_{k+1},\dots,x_n) \] 
with $f_1\geq 0$ and $f_2\geq 0$ and subject to the constraints
\[ g=g_1(x_1,\dots,x_k)+g_2(x_{k+1},\dots,x_n)\leq 0 \] 
and $x_i\geq 0$ for $i=1,\dots,n$ occurs at a maximum point for each of the functions $f_i$ subject to $x_i\geq 0$ and $g_j\leq (-1)^j\sigma, j=1,2$ for some constant $\sigma\in\R$. 
\end{lemma}

\begin{proof}[Proof of Lemma \ref{lem:split_opt}]
First we consider the problem without the constraints $x_i\geq 0$ for $i=1,\dots,n$. Using Lagrange multipliers, any critical point of the combined constrained optimization occurs at a point where $\grad f=\lambda \grad g$ and $g\leq 0$ for some value $\la\in \R$ of the auxiliary variable. (Note that $\la=0$ precisely when $g<0$ at the critical point.) This implies 
\begin{equation}
\label{eq:LastStand1}
f_1 \grad  f_2 +f_2 \grad f_1=\lambda (\grad g_1+\grad g_2).
\end{equation}
By separation of variables, \eqref{eq:LastStand1} becomes $f_2 \grad f_1=\lambda \grad g_1$ and $f_1 \grad f_2=\lambda \grad g_2$, where the gradients are with respect to $x_1,\dots,x_k$ and $x_{k+1},\dots,x_n$ respectively. Therefore setting $\la_1=\frac{\la}{f_2}$ and $\la_2=\frac{\la}{f_1}$, where $f_1$ and $f_2$ are evaluated at the critical point, we have $\grad f_1=\lambda_1 \grad g_1$ and $\grad f_2=\lambda_2 \grad g_2$. Evaluating $g_1$ and $g_2$ at the critical point yields $g_1=-\sigma$ and $g_2\leq \sigma$ for some value $\sigma\in \R$. In other words, the critical point for the combined optimization problem occurs at a pair of critical points for the separate constrained optimization problems.

Adding in the constraints that each $x_i\geq 0$ amounts to eliminating some critical points that are no longer admissible, and adding new generalized critical points that occur on the boundary. The remaining critical points will be critical for the separated constrained optimizations where we include the constraints that each $x_i\geq 0, 1\leq i\leq k$ 
and each $x_i\geq 0, k+1\leq i\leq n$, respectively. To check the additional boundary conditions, if one or more $x_i=0$, then we apply the same procedure to the objective function with fewer variables, where the corresponding $x_i$ are substituted with $0$. By induction, all of the generalized critical points for the combined optimization are generalized critical points for the two separated optimizations. Noting that the $f_i$ are nonnegative, any maximum of the combined problem will correspond to a maximum of the pair of objective functions under the given constraints for some choice of $\sigma$ since any other values at individual critical points, for the same $\sigma$, will only yield a smaller value of their product.
\end{proof}

Using Lemma \ref{lem:split_opt}, we can split the optimization problem into two 
smaller-dimensional optimization problems. In the next lemma, we solve the 
smaller-dimensional optimization problem. 

\begin{lemma}\label{lem:each_opt}
Given $0\leq \sigma\leq 1$, suppose that one of the following statements holds:
\begin{itemize}
\item[(a)]
$d=1$ and $k\geq 3$ or $\sigma\leq 2\sqrt{2}-2 \approx 0.8284$.
\item[(b)] $d=2$ and $k\geq 4$ or $\sigma\leq \frac{1}{\sqrt{2}} \approx 0.7071$. 
\item[(c)] $d=4$ and $k\geq 4$ or $\sigma\leq \frac{1}{7} \left(1+2\sqrt{2}\right)\approx 0.5469$. 
\item[(d)] $d=8$ and $k\geq 6$ or $\sigma\leq\frac{1}{23} \left(3+4 \sqrt{2}\right) \approx 0.3764$. 
\end{itemize}
Then the maximum of 
\[ f(x_1,\dots,x_k)=\frac{\sqrt{x_1\dots x_k}}{\prod_{i=1}^k (1-x_i+(d-1)x_{k+1-i})} \] 
subject to the constraint
\begin{equation}
\label{eq:constraint1}
 g(x_1,\dots,x_k)=\sum_{i=1}^k x_i\leq \sigma \quad \text{ and }x_i\in [0,\sigma]\text{ for all }i=1,\dots, k
\end{equation}
occurs when 
\[ x_1=\dots=x_k=\frac{\sigma}{k}. \] 
\end{lemma}

\begin{proof}[Proof of Lemma \ref{lem:each_opt}]
Note that $f$, $g$ and the domain $[0,\sigma]^k$  are all invariant under any permutation of $x_1,\dots,x_k$ if $d=1$. If $d>1$ then $f$ is only invariant under all permutations of the pairs $x_i,x_{k+1-i}$ for $i=1,\dots,\frac{k}{2}$ and is also invariant under all exchanges $x_i \leftrightarrow x_{k+1-i}$. Hence the set of points achieving the maximum value also possesses these symmetries as well. In particular, if the maximum occurs at a unique point in the admissible domain, then $x_i=\frac{\sigma}{k}$ for all $i$ as desired.

If any $x_i=0$ then $f=0$ provided either $\sigma<1$ or $\sigma=1$ and $k>3$. If $\sigma=1$ and $k=3$ and $d\geq 2$ then again $f=0$. In the remaining $k=3,d=1,\sigma=1$ case, then $f=0$ unless two of the three $x_i$ approach $0$. In that event, the maximum limit on the boundary is $f(x_1,x_2,1-x_1-x_2)\to 1/2$ occuring when any two of the $x_i$ approach $0$ at the same rate. However, this value is less than the value $f(1/3,1/3,1/3)=\sqrt{\frac{27}{64}}$, and so no maximum ever occurs at a boundary point of the domain.

It remains to show that a maximum occurs at a unique point in the open admissible domain of $(0,\sigma)^k$. We treat the $d=1$ case first. 

\paragraph{Case 1 ($d=1$): }  Since $\log$ is increasing, maximizing $f$ is equivalent to maximizing $\log (f^2)$. Setting
\[ H(x)=\frac{x}{(1-x)^2} \] 
and using Lagrange multipliers, we must find the critical point solutions of
\begin{equation}\label{eq:F-separate}
\grad \log H(x_i)=\frac{1}{x_i}+\frac{2}{1-x_i}=\la
\end{equation}
for $\la=0$, or for any $\la$ with $g=\sigma$, $x_i = \sigma_i$ for some $0\leq \sigma_i\leq \sigma$, and $\sum_{i=1}^k \sigma_i = \sigma$. We also note that the value of $\la$ is independent of $i$. From (\ref{eq:F-separate}), we obtain
\[ \lambda x_i^2 + (1-\lambda)x_i + 1 = 0 \]
and thus there are only two such generalized critical points
\[ c_{\pm}(\la)=\frac{\la-1\pm\sqrt{\la^2-6\la +1}}{2 \la}. \]
These are both real and nonnegative only when $\la\geq 3+2\sqrt{2}.$ Also, if only one is positive and real, then that is the value taken on by each $x_i$ which must then be $\frac{\sigma}{k}$. For each $i$, we have either $\sigma_i = c_+(\la)$ or $\sigma_i = c_-(\la)$. Additionally, we have $\sum_{i=1}^k \sigma_i = \sigma \leq 1$. We assert that for each $i$, we must have $\sigma_i = c_-(\la)$ if either $k \geq 3$ or $\sigma \leq 2\sqrt{2} -2$. If $k \geq 3$, we have 
\begin{equation}\label{eq:Match1}
2 c_-(\la)+c_+(\la)=\frac{3 (\lambda-1) -\sqrt{(\lambda -6) \lambda +1}}{2 \lambda }.
\end{equation}
For $\la \geq 3+2\sqrt{2}$, we see that (\ref{eq:Match1}) is not less than $3(\sqrt{2}-1) \approx 1.2426 >1$. As $c_+\geq c_-$, we have that
\[ 3c_+(\la) \geq 2c_+(\la) + c_-(\la). \] 
Consequently, we also have $3c_+(\la)>1$. As both violate our constraint on $\sigma$, we conclude that for $k\geq 3$ all the $x_i=c_-(\la)=\frac{\sigma}{k}$. In the second case that $\sigma \leq 2\sqrt{2} -2$, we see that $c_+(\la)\geq 2\sqrt{2}-2$ for $\la\geq 3+2\sqrt{2}$. In particular, if $c_+(\la)$ occurs for any $i$, then $\sigma$ would exceed this upper bound. Therefore, we conclude that $x_i=c_-(\la)=\frac{\sigma}{k}$.

\paragraph{Case 2 ($d=2,4,8)$ : }  For the $d=2,4,8$ cases we first apply Lemma \ref{lem:split_opt} successively, to show that the maximum occurs at a maximum of 
\[ H(x_i,x_{k+1-i})=\frac{x_i x_{k+1-i}}{(1-x_i+(d-1)x_{k+1-i})^2(1-x_{k+1-i}+(d-1)x_i)^2} \] 
subject to $x_i+x_{k+1-i}\leq \sigma_i$ for some $0\leq \sigma_i\leq \sigma$ and each $i=1,\dots,\frac{k}{2}$. The function $H(x_i,\sigma_i-x_i)$ has two critical points up to the symmetry $x_i\leftrightarrow \sigma_i-x_i$, namely $c_0(\sigma_i)=\frac{\sigma_i}{2}$ and the symmetric pair,
\[ c_\pm(\sigma_i)=\frac{\sigma_i}{2}\pm\frac{\sqrt{\sigma_i^2 \left(d^2+4 d-4\right)-4\sigma_i (d-2)-4}}{2 d}. \]
The critical points $c_\pm(\sigma_i)$ are real and distinct from $c_0(\sigma_i)$ only when 
\[ \sigma_i> \frac{2 \left(\sqrt{2} d+d-2\right)}{d (d+4)-4}. \] 

We again assert that under the conditions on $k$ and $\sigma$ in each of these case, the maxima for each $i$ occur at the critical points $c_0(\sigma_i)$. To begin, we first observe that
\[ H(c_0(\sigma_i),c_0(\sigma_i))\leq H(c_+(\sigma_i),c_-(\sigma_i)) \]
for 
\[ \sigma_i\geq \frac{2 \left(\sqrt{2} d+d-2\right)}{d (d+4)-4}. \] 
When $d=2$ or $d=4$, we also have 
\begin{equation}\label{eq:Match2}
\frac{2 \left(\sqrt{2} d+d-2\right)}{d (d+4)-4}>1/2.
\end{equation}
As $\sigma_i \leq \sigma \leq 1$, from (\ref{eq:Match2}) we see that there can only be at most one of the $x_i,x_{k+1-i}$ pairs that equal $c_+(\sigma_i),c_-(\sigma_i)$, and the rest take on values $c_0(\sigma_i),c_0(\sigma_i)$. Assuming this pair occurs at $i=1$, for any $\sigma'$ with $\sigma_1 < \sigma' \leq \sigma$, a straightforward calculation shows that
\[ H(c_0(\sigma'/2),c_0(\sigma'/2))^2\geq H(c_+(\sigma_1),c_-(\sigma_1))H(c_0(\sigma'-\sigma_1),c_0(\sigma'-\sigma_1)) \]
for all values of 
\[ \sigma_1> \frac{2 \left(\sqrt{2} d+d-2\right)}{d (d+4)-4}. \]
Hence, when $k\geq 4$, all the pairs take on values of the form $(x_i,x_{k+1-i})=(\frac{\sigma_i}{2},\frac{\sigma_i}{2})$ at the maximum.

When $k\geq 6$ and $d=8$, we have
\[ \frac{2 \left(\sqrt{2} d+d-2\right)}{d (d+4)-4}>\frac13, \] 
and so we can have at most two $x_i$ equal to the $c_+$ critical point. We have for $\sigma_1+\sigma_2\leq \sigma'$ that
\[ H(c_0(\sigma'/3),c_0(\sigma'/3))^3\geq H(c_+(\sigma_1),c_-(\sigma_1))H(c_+(\sigma_2),c_-(\sigma_2))H(c_0(\sigma'-\sigma_1-\sigma_2),c_0(\sigma'-\sigma_1-\sigma_2)) \]
for all values of 
\[ \sigma_1,\sigma_2> \frac{2 \left(\sqrt{2} d+d-2\right)}{d (d+4)-4}. \] 
Similarly,
\[ H(c_0(\sigma'/3),c_0(\sigma'/3))^3\geq H(c_+(\sigma_1),c_-(\sigma_1))H(c_0(\sigma_2),c_0(\sigma_2))H(c_0(\sigma'-\sigma_1-\sigma_2),c_0(\sigma'-\sigma_1-\sigma_2)) \]
for all values of 
\[ \sigma_1> \frac{2 \left(\sqrt{2} d+d-2\right)}{d (d+4)-4} \] 
with $\sigma_1+\sigma_2\leq \sigma'\leq \sigma \leq 1$. Hence we have that the only points that achieve the constrained maximum satisfy 
\[ x_i=x_{k+1-i}=c_0(\sigma_i)=\frac{\sigma_i}{2}. \] 
Now we will show that all of the $\sigma_i$ are the same. When $k=2$ we already have $x_1=x_2$, so we will assume $k>2$. Setting $x_i=x_{k+1-i}$, there are only two conjugate generalized critical points for
$\grad \log H(x,x)=\la$. Namely
\begin{equation}\label{eq:Match3}
c_{\pm}(\la)=\frac{-2(d-2)-\la \pm \sqrt{12 (d-2) \la+4 (d-2)^2+\la^2}}{2 (d-2) \la}.
\end{equation}
In particular, since (\ref{eq:Match3}) is independent of $i$, these are the only two possible values of $\sigma_i/2$. In order for $c_{+}$ and $c_{-}$ to both be potential choices they must both be positive, but that occurs if and only if $d=8$ and $-36 + 24 \sqrt{2}<\la < -2$. Under these constraints, if $k$ is even then $2c_+ + 2c_->1$. In particular, there cannot be two distinct values of $\sigma_i$. 

If $k$ is odd then there is a single central term in the product $f$ of the form $H(x_k,x_k)^{1/2}$. Note that the critical points for $\grad \frac{1}{2}\log H(x,x)=\la$ are simply $c_{\pm}(2\la)$. For $\la < 0$, we have $c_-(\la)> c_+(\la)$. However, $2c_+(\la) + c_+(2\la)>1$  when $-36 + 24 \sqrt{2}<\la < -2$. Hence, for $k$ odd we also cannot have distinct $\sigma_i$, completing the lemma.
\end{proof}

With the above lemmas in hand, we can prove Proposition \ref{prop:factor}. 

\begin{proof}[Proof of Proposition \ref{prop:factor}]
Using Lemma \ref{lem:split_opt} with explicit functions $f_1=\frac{x_1\dots x_j}{\prod_{i=1}^j (1-x_{i})^2}$ and $f_2=\frac{x_{j+1}\dots x_p}{\prod_{i=j+1}^{p}(1-x_i+(d-1)x_{nd+1-i})^2}$, we apply Lemma \ref{lem:each_opt} to $f_1$ with $k=j$ and $f_2$ with $k=p-j=nd-2j$ separately. When $d=1$, both terms combine and the result follows from Lemma \ref{lem:each_opt}. In the case $d=2,4$ and $p-j\geq 4$, or $d=8$ and $p-j\geq 6$, then the result follows trivially if $j=0$ or $j=1$ or $p-j=0$ or $1$. The same holds if both $p-j\geq d$ (or $p-j\geq 6$ if $d=8$) and $j\geq 3$. The remaining three cases are when $d=2,4$ or $8$ and $p-j=2$, or $d=8$ and $p-j=4$, or $j=2$. 

Write $\sigma_1=\sum_{i=1}^j x_i$ and $\sigma_2=\sum_{i=j+1}^p x_i$. In order to apply Lemma \ref{lem:each_opt}, we must ensure that $\sigma_2$ satisfies the constraints for each of the cases $d=2,4,8$. For any values $x,y\in(0,1)$, 
\[ \frac{x}{1-x}\geq \frac {x}{1-x+(d-1)y}. \] 
Since each of these is monotone increasing in $x$, the average value of the $x_i$ for $i<j$ is always 
larger than the average of the $x_i$ for $i>j$, otherwise by decreasing the largest 
$x_i$ for $i>j$ by a small amount and adding it to the smallest $x_i$ for $i\leq j$ we 
would obtain an increase in the objective function $f=f_1f_2$. Hence, 
\[ \sigma_1\geq \pr{\frac{j}{p-j}}\sigma_2. \] 
Moreover since $\sigma_1+\sigma_2\leq 1$, we have 
\[ \sigma_2\leq 1-\sigma_1\leq 1- \pr{\frac{j}{p-j}}\sigma_2 \] 
or $\sigma_2\leq \frac{p-j}{p}$. Recall that $p-j=nd-2j$, so $\sigma_2\leq \frac{nd-2j}{nd-j}$. If $p-j=2$, then 
\begin{equation}\label{eq:Match4}
\sigma_2\leq \frac{2}{2+j}.
\end{equation}
For $j\geq 2$, we see that (\ref{eq:Match4}) is less than $\frac1{\sqrt{2}}$ and $\frac17 (1+2\sqrt{2})$. In particular, if $p-j=2$ and $d=2$ or $4$, we can apply Lemma \ref{lem:each_opt}. If $d=8$ and $p-j=2$, then $j=8n-2>4n$ contrary to our hypothesis that $j < \frac{nd}{2}$. Likewise, if $p-j=4$ and $d=8$, then $j=8n-4> 4n$ when $n\geq 2$, so again this is contrary to our hypothesis. The remaining case is $j=2$ and $p-j\geq 4$. In this case, we have seen that the conditions on $\sigma_2$ needed for Lemma \ref{lem:each_opt} hold (for $d=8$, we have $p-j \geq12$) and so by Lemma \ref{lem:each_opt} the maximum for the $f_2$ factor occurs at
\[ x_{j+1}=x_{j+2}=\dots=x_{p}=\frac{\sigma_2}{p-j}. \] 
Writing $\sigma_2=\sigma$, $\sigma_1=1-\sigma$, $x_1=x$ and $x_2=1-\sigma-x$, the objective function becomes the 
two variable function 
\[ f(x,\sigma)=\frac{x ((p-2) \sigma )^{p-2} (1-x-\sigma) ((d-2) \sigma +p-2)^{4-2 p}}{(x-1)^2 (\sigma +x)^2} \]
The critical points in $x$ are $c_0=\frac{1-\sigma }{2}$ and 
\[ c_\pm=\frac{1}{2} \left(1-\sigma \pm \sqrt{\sigma ^2-6 \sigma +1}\right) \] 
with corresponding values:
\[f(c_0,\sigma)=\frac{4 (\sigma -1)^2 ((p-2) \sigma )^{p-2} ((d-2) \sigma +p-2)^{4-2 p}}{(\sigma +1)^4} \]
and
\[ f(c_\pm,\sigma)=\frac{1}{4} (p-2)^{p-2} \sigma ^{p-3} ((d-2) \sigma +p-2)^{4-2 p}. \]
However at the maximum of $\sigma$ in the second case is at the point 
\[ \sigma =\frac{(p-3) (p-2)}{(d-2) (p-1)}=\frac{(nd-5) (nd-4)}{(d-2) (nd-3)}>1 \] 
when $p=nd-2$. In particular the maximum at the second critical point pair in the valid domain of $\sigma\in[0,1]$ occurs when $\sigma=1$, or when $\sigma_1=1-\sigma=0$, but then $x_1=x_2=0$ when $\sigma_1=0$. The last possibility is when the maximum in $x$ occurs at the first critical point, but that is when $x_1=x_2=\frac{\sigma_1}{2}$. 
\end{proof}

%---------------------------------------------------------------------------
%---------------------------------------------------------------------------
%---------------------------------------------------------------------------
%---------------------------------------------------------------------------
%%%%%%%%% APPENDIX  B%%%%%%%%%%%%%%%%
%---------------------------------------------------------------------------
%---------------------------------------------------------------------------
%---------------------------------------------------------------------------
%---------------------------------------------------------------------------
\subsection{Proof of Proposition \ref{prop:Pest}}
\label{AppendixSectionB}\label{AppB}

\begin{proof}[Proof of Proposition \ref{prop:Pest}]
First we show that, apart from the exceptional case $j=1$, $d=2$, and $n=2$, the value of $\la$ at the maximum occurs in the open interval $\left(0,\frac{1}{p-j}\right)$. Indeed, we know $\la\geq 0$ and $(p-j)\la=1-j\sigma\leq 1$. However, at the endpoints we have $P(0) = 0 = P(\frac{1}{p-j})$, while $P$ is clearly positive in between. When $j=1$, $d=2$, $n=2$, $P(\la)=\frac12\sqrt{1-2 \la}$ and the maximum $\frac12$ occurs at $\la=0$.

To find the $\lambda$ for which $P(\lambda)$ is maximal, we consider the critical points of $P^2(\lambda)$ which are also those of $P(\lambda)$. Since $P$ is a positive function, the maximum of $P$ and $P^2$ occurs at the same critical points. By direct calculation we have 
\begin{align}\label{eq:Pder}
\frac{\partial P^2(\la)}{\partial \la}=\frac{(p-j) j^j (1-(p-j) \la)^{j-1} \la^{p-j-1} Q(\la)}{(1+(d-2) \la)^{1+2 (p-j)} (j-1+(p-j) \la)^{1+2 j}},
\end{align}
where $Q(\la)$ is the cubic polynomial in $\la$ given by
\begin{align*}
Q(\la)=p (d-2)&(p-j) \la^3 + \left(j (d-2 - 2 j(d-1))+p (d (j-2)+j+4)-p^2\right) \la^2+\\
&\left( p(2-j) -j(d+1)+d-2\right) \la +(j-1).
\end{align*}
Note that the denominator of \eqref{eq:Pder} is never zero since $\la>0$, 
\[ (p-j)\la=1-j\sigma<1, \] 
and we are assuming $p\geq j\geq 1$. For the same reason, the numerator of \eqref{eq:Pder} vanishes precisely when $Q(\la)$ does. 

We note that the polynomial $Q(\la)$ is negative as $\la$ tends to $-\infty$, even when $d=2$. At $\la=0$, $Q(0) = j-1$, which is positive provided $j>1$. The first derivative at $\la=0$ when $j=1$ is $Q'(0) = p-3$, which is positive except when $n=d=2$ and $j=1$. In this exceptional case, $Q(\la) = -2\la^2$, and therefore the maximum of $P(\lambda)$ occurs when $\la=0$ and $\sigma=1$, as we already showed. In all other cases, there is a root in $(-\infty,0)$. Since $Q\pr{\frac{1}{p-j}} < 0$ and $Q$ is cubic, there is also exactly one root in $(0,\frac{1}{p-j})$. Moreover, as $Q(\la)$ goes from positive to negative around this root, it corresponds to the local, and hence global, maximum of $P(\la)$.

We further note that every other factor in the numerator and denominator of $\frac{\partial P^2(\la)}{\partial \la}$ remains positive in the interval $\left(0,\frac{1}{p-j}\right)$ and so the sign of the derivative hinges only on the sign of $Q(\la)$. Moreover, any $\la>0$ with $Q(\la)<0$ will be larger than the root where the maximum occurs. At $\la=\frac{1}{p}$ we have,
$Q\left(\frac{1}{p}\right) = -\frac{2 (d-1) j^2}{p^2}$. This will always be negative whenever $d>1$. 

On the other hand, $\sigma=\frac{2}{p}$ corresponds to $\la=\frac{2}{p}-\frac{1}{p-j}$, and we have
\begin{equation}
\label{eq:Q-Func}
Q\left(\frac{2}{p}-\frac{1}{p-j}\right)=\frac{j^2 \left(-p^2 (3 d+j)+2 p (5 d j+d-3 j-2)-4 j (2 (d-1) j+d-2)+p^3\right)}{p^2 (p-j)^2}.
\end{equation}
Under our assumptions on $d$, $j$, and $n$ given in \eqref{eq:hypoth}, the above \eqref{eq:Q-Func} is never negative, except in the cases given in \eqref{eq:hypoth2}. Apart from these special cases, the value of $\la$ which achieves the maximum of $P(\la)$ lies in the interval $(\frac{2}{p}-\frac{1}{p-j},\frac{1}{p})$. Therefore the corresponding maximizing value for $\sigma=\frac{1+(p-j)\la}{j}$ lies in the interval $(\frac{1}{p},\frac{2}{p})$.

Since $P(\la,\sigma)$ is increasing in $\la$ and $\sigma$ separately, we can use the values $\la=\frac{1}{p}$ and $\sigma=\frac{2}{p}$ in place of the true maximizing values to obtain an over-estimate. That over-estimation gives the first desired estimate,
\[ P(\la) \leq P\left(\frac{1}{p},\frac{2}{p}\right)=\frac{2^{j/2} p^{p/2}}{(p-2)^j (p+d-2)^{p- j}}.\]
For the five special cases enumerated in \eqref{eq:hypoth2}, we can compute the maxima directly, and then form rational valued over-estimates as above. When $n=2$ and $j=1$, the root of $Q(\la)$ maximizing $P$ is  
\begin{equation}\label{eq:AppB1}
\la=\frac{2 (d-2)}{3 + d (-7 + 3 d) + \sqrt{-7 + d (22 + d (-17 + d (2 + d)))}}.
\end{equation}
We can approximate (\ref{eq:AppB1}) as follows.
For $d=2$ we already noted that the maximum is $P(0)=\frac12$. 
For the $n=2,d=4,j=1$ case, this root lies in $(\frac{4}{37},\frac19)$, so $P(\la)<P(\frac19,\frac{13}{37})$.
For the $n=2,d=8,j=1$ case, this root lies in $(\frac{1}{17},\frac{12}{203})$, so $P(\la)<P(\frac{12}{203},\frac{3}{17})$.
Similarly, for the $n=2,d=8,j=2$ case, the first positive root of $Q(\la)$ lies in  $(\frac{1}{20},\frac{3}{50})$, so $P(\la)<P(\frac{3}{50},\frac{1}{5})$.
For the $n=2,d=8,j=3$ case, the first positive root of $Q(\la)$ lies in  $(\frac{1}{20},\frac{7}{125})$, so $P(\la)<P(\frac{7}{125},\frac{1}{6})$.
These values give all the exceptional estimates completing the proposition.
\end{proof}

Multiplying the exceptional estimates in Proposition \ref{prop:Pest} by $\frac{(\delta(\Gamma))^p}{p^{p/2}}$ we obtain the Jacobian estimates below:
\begin{align}\label{eq:ExceptionalCaseList}
(n=2,~d=2,~j=1): \quad\quad \abs{\Jac_{2\cdot 2-1}(F,x)} &\leq \frac{(\delta(\Gamma))^3}{6 \sqrt{3}} \notag \\
(n=2,~d=4,~j=1): \quad\quad \abs{\Jac_{4\cdot 2 - 1}(F,x)} &\leq \frac{(3)^5 (\sqrt{13})(\sqrt{37})(\delta(\Gamma))^7}{(\sqrt{7}) (2)^3 (7)^3 (11)^6} \notag \\
(n=2,~d=8,~j=1): \quad\quad \abs{\Jac_{8 \cdot 2 - 1}(F,x)} &\leq \frac{(2)^{13} (7)^6 (29)^7 (\sqrt{17})(\delta(\Gamma))^{15}}{(\sqrt{5}) (5)^{35} (11)^{14}} \\
(n=2,~d=8,~j=2): \quad\quad  \abs{\Jac_{8 \cdot 2 - 2}(F,x)} &\leq \frac{(3)^6 (5)^{13} (\delta(\Gamma))^{14}}{(2)^{29}(7)^7(17)^{12}} \notag \\
(n=2~d=8,~j=3): \quad\quad \abs{\Jac_{8 \cdot 2 - 3}(F,x)} &\leq \frac{(6\sqrt{6}) (5)^{12}(7) ^5(\delta(\Gamma))^{13}}{\sqrt{13} (13)^{6}(167)^{10}}. \notag
\end{align}

%---------------------------------------------------------------------------
%---------------------------------------------------------------------------
%%%%    THE BIBLIO    %%%%%%%%%%%%%%%%%%%
%---------------------------------------------------------------------------
%---------------------------------------------------------------------------

%---------------------------------------------------------------------------
%---------------------------------------------------------------------------


\begin{thebibliography}{9999}
\small


\bibitem[BCG95]{BCGRigid}
G.~Besson, G.~Courtois, and S.~Gallot, \textsl{Entropies et rigidit\'es des espaces localement sym\'etriques de courbure strictement n\'egative}. Geom. Funct. Anal. \textbf{5} (1995), 731--799.

\bibitem[BCG99]{BCGSchwartz}
G.~Besson, G.~Courtois, and S.~Gallot, \textsl{Lemme de Schwartz r\'{e}el et applications g\'{e}om\'{e}triques}. Acta Math. \textbf{183} (1999), 145--169.

\bibitem[BJ97]{BJ}
C.~J.~Bishop and P.~W.~Jones, \textsl{Hausdorff dimension and Kleinian groups}, Acta Math. \textbf{179} (1997), 1--39.

\bibitem[CP04]{CP}
G.~Carron and E. Pedon, \textsl{On the differential form spectrum of hyperbolic manifolds}, Ann. Scuola Normal. Sup. Pisa \textbf{5} (2004), 705--747.

\bibitem[CFM-2]{CFM}
C.~Connell, B.~Farb, and D.~B.~McReynolds, \textsl{The cohomological dimension spectrum of a group}, in preparation.

\bibitem[CFM-3]{Mathematica}
C.~Connell, B.~Farb, and D.~B.~McReynolds, \textsl{Gap Theorem Appendices Computations}, \url{http://pages.iu.edu/~connell/publications/GapTheoremAppendicesComputations.nb}, 2015.

\bibitem[Cor90]{Corlette}
K.~Corlette, \textsl{Hausdorff dimensions of limit sets. I.} Invent. Math. \textbf{102} (1990), 521--541.

\bibitem[Fie71]{Fiedler}
M.~Fiedler, \textsl{Bounds for the determinant of the sum of Hermitian matrices}, Proc. Amer. Math. Soc. \textbf{30} (1971), 27--31.

\bibitem[Gro82]{GBC}
M.~Gromov, \textsl{Volume and bounded cohomology}, Inst. Hautes \'Etudes Sci. Publ. Math. \textbf{56} (1982), 5--99.

\bibitem[Gro99]{Gromov}
M.~Gromov, \textsl{Metric structures for Riemannian and non-Riemannian spaces}, Birkh\"{a}user, 1999.

\bibitem[Kap09]{Kapovich}
M.~Kapovich, \textsl{Homological dimension and critical exponent of Kleinian groups}, Geom. Funct. Anal. \textbf{18} (2009), 2017--2054.

\bibitem[KK15]{Kim2}
S.~Kim and I.~Kim, \textsl{Simplicial volume, barycenter method, and bounded cohomology}, preprint \url{http://arxiv.org/abs/1503.02381}

\bibitem[Kos69]{Kostant}
B.~Kostant, \textsl{On the existence and irreducibility of certain series of representations}, Bull. Amer. Math. Soc. \textbf{75} (1969), 627--642.

\bibitem[LW15]{LafontWang}
J.-F.~Lafont and S.~Wang, \textsl{Barcycentric straightening and bounded cohomology}, preprint \url{http://arxiv.org/abs/1503.06369}.

\bibitem[Li92]{Li}
J-S.~Li, \textsl{Nonvanishing theorems for the cohomology of certain arithmetic quotients}, J. Reine Angew. Math. \textbf{428} (1992), 177--217.

\bibitem[Mar91]{Margulis}
G.~Margulis, \textsl{Discrete Subgroups of Semisimple Lie Groups}, Springer, 1991.

\bibitem[Mat62]{Matsushima}
Y.~Matsushima, \textsl{On the first Betti number of compact quotient spaces of higher-dimensional symmetric spaces}, Ann. of Math. (2) \textbf{75} (1962), 312--330.

%\bibitem[Mil76]{Mil76}
%J.~J.~Millson, \textsl{On the first Betti number of a constant negatively curved manifold}, Ann. of Math. \textbf{104} (1976), 235--247.

\bibitem[St02]{Starkov}
A.N. Starkov, \textsl{Vanishing of the first cohomologies for lattices in Lie groups}, J. Lie Theory \textbf{12} (2002), 449--460.

\bibitem[Rag72]{Raghunathan}
M.~S.~Raghunathan, \textsl{Discrete Subgroups of Lie Groups}, Springer, 1972.

\bibitem[Sul79]{Sullivan}
D.~Sullivan, \textsl{The density at infinity of a discrete group of hyperbolic motions}, Inst. Hautes \'Etudes Sci. Publ. Math. \textbf{50} (1979), 171--202.

\end{thebibliography}
\end{document}